\documentclass[sn-mathphys-num]{sn-jnl}% Math and Physical Sciences Numbered Reference Style 
\usepackage{graphicx}%
\usepackage{multirow}%
\usepackage{amsmath,amssymb,amsfonts}%
\usepackage{amsthm}%
\usepackage{mathrsfs}%
\usepackage[title]{appendix}%
\usepackage{xcolor}%
\usepackage{textcomp}%
\usepackage{manyfoot}%
\usepackage{booktabs}%
\usepackage{algorithm}%
\usepackage{algorithmicx}%
\usepackage{algpseudocode}%
\usepackage{listings}%
\usepackage[utf8]{inputenc}
\usepackage{babel}

\theoremstyle{thmstyleone}%
\theoremstyle{thmstyletwo}%
\newtheorem{remark}{Remark}%

\theoremstyle{thmstylethree}%

\newtheorem{thm}{Theorem}[section]

\newtheorem{lemma}[thm]{Lemma}

\newtheorem{prop}[thm]{Proposition}

\newcommand{\C}{\mathbb{C}}
\newcommand{\N}{\mathbb{N}}

\begin{document}

\title[Stabilization for KP-II without size restrictions]{Boundary Exponential Stabilization for the Linear KP-II equation without Critical Size Restrictions}


\author*[1]{\fnm{F.A} \sur{Gallego}}\email{fagallegor@unal.edu.co}

\author[2,3]{\fnm{J.R} \sur{Mu\~noz}}\email{juan.ricardo@ufpe.br}
\equalcont{These authors contributed equally to this work.}

\affil*[1]{\orgdiv{Departamento de Matem\'atica}, \orgname{Universidad Nacional de Colombia (UNAL)}, \orgaddress{\street{Cra 27, No. 64-60}, \city{Manizales}, \postcode{170003}, \state{Caldas}, \country{Colombia}}}

\affil[2]{\orgdiv{Departamento de Matem\'atica}, \orgname{Universidade Federal de Pernambuco}, \orgaddress{\street{S/N Cidade Universitaria}, \city{Recife}, \postcode{50740-545}, \state{PE}, \country{Brazil}}}
\affil[3]{ \orgdiv{Departament of Electrical Engineering and applied Computing}, \orgname{University of Dubrovnik}, \orgaddress{\street{Ćira Carića 4}, \city{Dubrovnik},  \postcode{20000}, \country{Croatia}}}


\abstract{In this paper, we delve into the intricacies of boundary stabilization for the linearized KP-II equation within the constraints of a bounded domain, a phenomenon known as ``critical length." Our primary aim is to design a feedback law that ensures the existence and exponential stabilization of solutions in the energy space, without length restrictions on the domain \( \Omega = (0, L) \times (0, L) \), \( L > 0 \). Furthermore, we examine the interaction between the drift term \( u_x \) under these constraints.
}




\maketitle

\section{Introduction}
The KP equation, short for the Kadomtsev-Petviashvili equation, holds great significance in the realm of partial differential equations. It emerges prominently when investigating nonlinear waves and solitons, with applications spanning water wave dynamics, plasma physics, and various theoretical physics and mathematical disciplines. This equation, independently discovered by B.B. Kadomtsev and V.I. Petviashvili in 1970 (see \cite{kp70}), takes center stage in the study of weakly nonlinear, long waves within dispersive media. Its general form is typically expressed as follows:
\begin{equation*}
u_t+ u_x +uu_x +u_{xxx} +\gamma\partial_x^{-1}(u_{yy})=0, \quad \gamma=\pm 1.
\end{equation*}
When $\gamma= -1$, we denote it as KP-I, and when $\gamma =1$, it's KP-II, inspired by the KP equation's convention. KP-II describes long gravity waves with subtle surface tension, while KP-I models capillary gravity waves, especially with strong surface tension. Notably, the KP-I equation also arises in the transonic limit (under appropriate scaling) of the defocusing nonlinear Schrodinger equation, as documented in \cite{jones}. 

In this paper, we explore boundary stabilization of the Kadomtsev-Petviashvili (KP) equation, focusing on control strategies and dissipative boundary conditions to ensure stability near domain boundaries. This is crucial in applications like fluid dynamics to prevent wave or soliton instability. Researchers use energy methods, feedback control, and simulations to achieve global solution existence and decay. The integrable nature of the KP equation adds to its intrigue, making it a challenging problem in partial differential equations and mathematical physics.

\subsection{Setting of the problem}
We focus on the boundary stabilization scenario for the KP-II equation in a bounded domain. Specifically, we investigate the implications of applying Neumann dissipation laws at the domain boundaries, aiming to understand their impact on achieving exponential stabilization and the interplay between temporal and spatial constraints. We consider parameters $L > 0$ for a positive length, $\Omega := (0, L) \times (0, L)$ to define a rectangular region in $\mathbb{R}^2$, and $t > 0$. Within this context, we explore the following closed-loop system to address these boundary stabilization challenges:

\begin{equation}\label{eq:KP}
\begin{cases}
		u_t(x,y,t) +  u_{x}(x,y,t)+  u_{xxx}(x,y,t) + \partial_x^{-1} u_{yy}(x,y,t)=0, \\ 
	u(0,y,t)=u(L,y,t)=0,\quad  u_y(x,L,t)=\beta u_x(x,L,t), \\
	u_x(L,y,t)=\alpha u_x(0,y,t), \quad u_y(x,0,t)=0,\\
	u(x,y,0)=u_0(x,y)
\end{cases}
\end{equation}
where $(x,y) \in \Omega$, $t >0$, the feedback gains $0<|\alpha|< 1$, $\beta>0$, the operator 
$\partial_x^{-1}$ is defined by  $\partial_x^{-1} \varphi(x,y,t)= \psi(x,y,t)$ such that $\psi(L,y,t)=0$ and $\psi_x(x,y,t)=\varphi(x,y,t)$.
It can be shown that the definition of operator $\partial_x^{-1}$ is equivalent to $$\partial_x^{-1} u(x,y,t) = -\int_x^L u(s,y,t)\,ds.$$
 In this context, our goal is to demonstrate  that the energy associated  to the system described by equation  \eqref{eq:KP} defined by
\begin{equation}\label{eq:KPen}
E(t) =\frac{1}{2} \int_{\Omega} u^2(t)\,dx\,dy,
\end{equation}
has an exponential decay.  To visualize the dissipation effects resulting from the boundary conditions, we multiply the first equation by $u$ and we integrate over $\Omega$,
$$\frac{d}{dt} E(t) = -\int_{\Omega} u(x,y,t) \left(u_{x}(x,y,t)+  u_{xxx}(x,y,t) + \partial_x^{-1} u_{yy}(x,y,t)\right)\,dx\,dy. $$
 Moreover, in order to control the term $\partial_x^{-1} u_{yy}$, we consider $u_y(x,y) = f_x(x,y)$ with $f(L,y)=0$. Integrating by parts, it follows that
\begin{multline*}
\frac{d}{dt} E(t) 
	 = \frac{1}{2} \int_0^L \left[ u_x^2(x,y,t) - u(x,y,t)u_{xx}(x,y,t) + \left(\partial_x^{-1} u_y(x,y,t) \right)^2 \right]_{x=0}^{x=L} \,dy 
		 \\
	\quad - \int_0^L\left[ u(x,y,t) f(x,y,t)\right]_{y=0}^{y=L} \,dx 	 
\end{multline*}
Note that the energy of equation already has a natural dissipations mechanism due the influence of transverse variations (along the $y$-axis) given by  $\left(\partial_x^{-1} u_y(0,y,t) \right)^2$. Now, by using homogeneous Dirichlet conditions in the $x$-variable would require two types of feedback dissipations at the Neumann boundary throughout the domain. When employing the boundary feedback law represented by equation \eqref{eq:KP}, energy dissipation occurs, at least in the context of the associated linear system:
\begin{equation}
\begin{aligned}
\frac{d}{dt} E(t) 
	& =  \frac{-(1-\alpha^2)}{2} \int_0^L u_x^2(0,y,t)\,dy
		- \frac{1}{2} \int_0^L \left(\partial_x^{-1} u_y(0,y) \right)^2\,dy \\
	&\quad -\beta\int_0^L u^2(x,L)\,dx \leq 0.
\end{aligned}
\end{equation}

This observation highlights the significance of boundary conditions as feedback damping mechanisms. Consequently, motivated by certain critical length phenomena in the KdV and/or the KP-II equations \cite{Rosier, capistrano2024}, the following questions arise:
\vspace{0.2cm}
\begin{center}
\textit{Does $ E(t) \rightarrow 0$ as $ t \rightarrow \infty$? If so, can we determine the decay rate? and how does one avoid the critical length
restrictions?}
\end{center}
\vspace{0.2cm}
The main result of this paper provides a positive answer for
these questions.

\subsection{State of the Art}

Significant advancements in control theory have focused on understanding how damping mechanisms impact energy dynamics in systems governed by partial differential equations, particularly in dispersive equations related to water waves in bounded domains. Well-established results demonstrate the exponential stability of equations like KdV~\cite{Zuazua2002}, Boussinesq-KdV~\cite{Pazoto2008}, and Kawahara~\cite{Araruna2012}, among others. These analyses use the Compactness-Uniqueness framework by J.L. Lions~\cite{Lions1988}. Noteworthy findings in~\cite{Cerpa2021} and~\cite{Capistrano2018} employ Urquiza and Backstepping approaches. These results utilize damping mechanisms either within the equations or at the boundaries to achieve control. 

In two noteworthy studies, \cite{ailton2021} and \cite{Panthee2011}, exponential decay was observed in the KP-II and K-KP-II equations. These investigations established equation regularity and well-posedness while showing that introducing a suitable damping term ($a(x,y)u(x,y,t)$) leads to exponential energy decay. Notably, the first exploration of stabilization techniques for two-dimensional nonlinear dispersive equations in bounded domains, following the methodology presented here, is found in \cite{Panthee2011}, which specifically examines the KP-II equation. In \cite{Panthee2011}, external damping defined by a non-negative function $a(x,y)$ results in exponential energy decay, even when $a$ vanishes on a compact set $\Theta \subset(\delta, L-\delta) \times(\delta, L-\delta)$ for some $\delta>0$. The proof relies on an observation inequality for the KP-II equation, established through the unique continuation principle in \cite{panthee} and energy estimates using multiplier techniques. In contrast, \cite{ailton2021} employs the classical dissipation-observability method and Bourgain's unique continuation principle \cite{bourgain}, utilizing an anisotropic Gagliardo-Nirenberg inequality to drop the $y$-derivative of the norm for a key estimate. In \cite{Capistrano2023}, the authors investigate the K-KP-II system with both a damping mechanism $a(x,y)u$ acting as a feedback-damping mechanism, as seen in \cite{ailton2021}, and an anti-damping mechanism, preventing energy decrease. 
To the best of our knowledge, this current study represents the inaugural achievement in boundary feedback stabilization pertaining to the Kadomtsev-Petviashvili II equation within a rectangular domain $\Omega=(0,L)\times (0,L)$, with any {restriction} in $L>0$. In contrast, recently in \cite{capistrano2024}, a boundary dissipation mechanism was constructed that generates restrictions on the length of the rectangle. In particular, the critical phenomenon appears, characterized by a nonempty discrete set, specifically 
\begin{multline*}
\mathcal{R} = \left\lbrace \frac{\pi}{4n}\sqrt{(3m_1 + 2m_2 + m_3)((m_1 -m_3)^2  -4m_2^2)(m_1+2m_2 +3m_3)} : \quad \right. \\
\left. 
n, m_1, m_2, m_3 \in \mathbb{N},\ \text{with} \ \ |m_1-m_3|> 2m_2>0.\right\rbrace.
\end{multline*}

\subsection{Notations and Main Result}

Before presenting answers to stabilization question, let us introduce the functional space that will be necessary for our analysis.  Given $\Omega \subset \mathbb{R}^2$, let us define $X^k(\Omega)$ to be the Sobolev space
\begin{multline*}
X^k(\Omega):=
\left\{
\varphi \in H^k(\Omega) \colon {\partial_x^{-1}} \varphi(x,y)= \psi(x,y)\in H^k(\Omega),  \psi(L,y)=0, \right. \\
\left. \partial_x\psi(x,y)=\varphi(x,y)
\right\}
 \end{multline*} 
endowed with the norm
$
\left\lVert \varphi \right\rVert_{X^k(\Omega)}^2 = \left\lVert \varphi \right\rVert_{H^k(\Omega)}^2 + \left\lVert {\partial_x^{-1}} \varphi \right\rVert_{H^k(\Omega)}^2.$ Let us denote $\mathcal{H}=L^2(\Omega)$ and we also define the normed space $H_x^k(\Omega)$ as
\begin{equation}\label{eq:Hxk}
H_x^k(\Omega):=
	\left\lbrace%
	 	\varphi \colon \partial_x^j \varphi \in L^2(\Omega),\ \text{for } 0\leq j \leq k
	 \right\rbrace 
\end{equation}
with the norm $\left\lVert \varphi \right\rVert_{H_x^k(\Omega)}^2 =\sum_{j=0}^k \left\lVert \partial_x^j \varphi\right\rVert_{L^2(\Omega)}^2$. Similarly, we consider the space
\begin{multline*}
X_{x}^{k}(\Omega):=\left\{
\varphi \in H_x^k(\Omega) \colon {\partial_x^{-1}} \varphi(x,y)= \psi(x,y)\in H_x^k(\Omega), \psi(L,y)=0, \right. \\
\left.\partial_x\psi(x,y)=\varphi(x,y) 
\right\}
\end{multline*}
with norm $\left\lVert \varphi \right\rVert_{X_x^k(\Omega)}^2 = \left\lVert \varphi \right\rVert_{H_x^k(\Omega)}^2 + \left\lVert {\partial_x^{-1}} \varphi \right\rVert_{H_x^k(\Omega)}^2.$
Finally,   $H_{x0}^k(\Omega)$ will denote the closure of $C_0^\infty(\Omega)$ in $H_x^k(\Omega)$. 

The next result will be repeatedly referenced in this article, and it is known as the anisotropic Gagliardo-Nirenberg inequality:
\begin{thm}[{\cite[Theorem~15.7]{besov79}}]\label{thm:Besov}
Let $\beta$ and $\alpha^{(j)}$, for $j = 1,\dots, N$, denote $n$- dimensional multi-indices with non-negative-integer-valued components. Suppose that $1 < p^{(j)} < \infty$, $1 < q < \infty$, $0 < \mu_j < 1$ with
$$
\sum_{ j = 1 }^N \mu_j = 1, \quad
\frac{1}{q} \leq \sum_{j = 1}^N \frac{\mu_j}{p^{(j)}},\quad \beta - \frac{1}{q} = \sum_{j=1}^N \mu_j\left(\alpha^{(j)} - \frac{1}{p^{(j)}}\right).
$$
Then, for $f(x)\in C_0^\infty(\mathbb{R}^n)$,
$
\left\lVert
	D^\beta f
\right\rVert_{q} 
\leq C \prod_{j=1}^N 
\left\lVert 
	D^{\alpha^{(j)}}f
\right\rVert_{p^{(j)}}^{\mu_j},
$
 where for non-negative multi-index  $\beta = (\beta_1,\dots, \beta_N)$  we denote $D^\beta$ by 
$D^\beta = D_{x_1}^{\beta_1}\dots D_{x_n}^{\beta_n}$
and 
$ D_{x_i}^{\beta_i} = \frac{\partial^{\beta_i}}{\partial x_i^{k_i}}$.
\end{thm}

The next result asserts that \eqref{eq:KP} is globally uniformly exponentially stable in
\(\mathcal{H}\). It means that the decay rate \(\theta\)  is independent of intial data. Thus, the main results of this work provide an answer to the stabilization problem proposed at the beginning of this paper.  To summarize:
\begin{thm}[Global Uniform Stabilization]\label{11th:ESLyap_2}
		Let $L>0$. Then, for every initial data 
		$u_0\in L^2(\Omega)$ the energy associated to closed-loop system \eqref{eq:KP} defined by~\eqref{eq:KPen} decays exponentially. More precisely, there exists two positives constants $\theta$ and $\kappa$ such that $$E(t) \leq \kappa E(0)e^{-2\theta t}, \quad t >0.$$
	\end{thm}
%Please note that the decay rate $\theta$ in Theorem~\ref{11th:ESLyap_2} remains independent of the initial value's magnitude, $u_0$, within $L^2(\Omega)$. Consequently, we can conclude that the system \eqref{eq:KP} exhibits uniform exponential stability. 
To prove Theorem \ref{11th:ESLyap_2}, we'll employ the J.-L. Lions' compactness-uniqueness argument \cite{Lions1} to simplify the task to proving an observability inequality.

The structure of this work is as follows:  In Section~\label{s:2}, we discuss the well-
posedness of the system and some \textit{apriori} estimates. Section~\ref{s:3} explores the system's asymptotic behavior with the boundary dissipative mechanisms, without critical length restrictions. Finally, Section \ref{sec:final} is dedicated to further comments and open problems.

\section{{Well-posedness} and Main Estimates}\label{s:2}
%
%In this section, we delve into the findings regarding exponential stabilization when confronted with a dual dissipation mechanism. This mechanism operates within the Neumann Boundary Condition and the homogeneous Dirichlet boundary condition specifically along the $x$-variable. To accomplish this, we utilize an observability inequality.
%
%\subsection{Preliminar estimates}

In this context, we initiate a sequence of linear estimates that will be subsequently applied. Initially, we utilize classical semigroup theory in connection with the linear system \eqref{eq:KP}.

%{\color{gray}
%%%%%%%%%%%%%%%%%%
%%%%%%% ENERGY %%%%%% 
%%%%%%%%%%%%%%%%%%
%Multiplying by $u(x,y,t)$ and integrating by parts in $(0,L)\times(0,L)$. We define $f(x,y)$ such that $f_x(x,y) = u_y(x,y)$ with $f(L,y)=0$
%\begin{equation}
%\begin{aligned}
%0 & = \int_{\Omega} u (u_t + u_{xxx} + \gamma \partial_x^{-1} u_{yy})\,dx\,dy \\
%& = \frac{1}{2} \frac{d}{dt} \int_0^L \int_0^L u^2(x,y)\,dx\,dy
%	- \frac{1}{2} \int_{\Omega} (u_x^2)_x\,dx\,dy 
%	+ \int_0^L \left[u(x,y)u_{xx}(x,y)\right]_{x=0}^{x=L}\,dy \\
%&	\quad + \frac{\gamma}{2} \int_0^L f^2(0,y)\,dy 
%	- \int_0^L u(x,0) f(x,0)\,dx.
%\end{aligned}
%\end{equation}
%This implies
%\begin{equation}
%\begin{aligned}
%\frac{d}{dt} E(t) 
%	& = \frac{1}{2} \int_0^L u_x^2(L,y)\,dy - \frac{1}{2} \int_0^L u_x^2(0,y)\, dy \\
%	& \quad - \int_0^L u(L,y)u_{x}(L,y)\,dy + \int_0^L u(0,y)u_{xx}(0,y)\,dy\\
%	& \quad { \color{red} - \frac{\gamma}{2} \int_0^L \left(\partial_x^{-1} u_y(0,y)\right)^2\,dy }\\
%	& \quad - \int_0^L u(x,L) \partial_x^{-1} u_y(x,L)\,dx + \int_0^L u(x,0)\partial_x^{-1}u_y(x,0)\,dx
%\end{aligned}
%\end{equation}
%}

By using the notation of the operator $A$  we
rewrite the system \eqref{eq:KP} as an abstract evolution equation $$\frac{du}{dt}=A u,\quad u(0)=u_0$$
where $A \colon D(A)\subset \mathcal{H}\to \mathcal{H}$ is defined by
\begin{equation}\label{eq:MUA}
Au=-u_x 
		- u_{xxx}- {\partial_x^{-1}} u_{yy} 
\end{equation}
on the dense domain given by
\begin{align*}
D(A):=
	\left\lbrace
			u \in H_x^3(\Omega)\cap X^2(\Omega) \colon
			 u(0,y)=u(L,y)=0,\,  u_y(x,L)=\beta u_x(x,L), \right. \\
			\left.  u_x(L,y)= \alpha u_x(0,y), \,u_y(x,0)=0	 
	\right\rbrace,
\end{align*}
with $|\alpha| < 1, \beta >0$. In order to ascertain the soundness of the linear system, we introduce several valuable findings.
\begin{lemma}\label{le:MUAadj}
The operator $A$ is closed and the adjoint $A^\ast\colon D(A^\ast)\subset \mathcal{H}\to \mathcal{H}$ is given by $
A^\ast v =v_x+
		 v_{xxx}+ \left(\partial_x^{-1}\right)^{\ast} v_{yy}$
with dense domain
\begin{align*}
D({A}^\ast):=
	\left\lbrace
			v\in H_x^3(\Omega)\cap X^2(\Omega) \colon
			v(0,y)=v(L,y)=0,  v_y(x,L)=-\beta v_x(x,L), \right. \\
			\left. 
			 v_x(0,y)=\alpha v_x(L,y), 	\,\, v_y(x,0)=0  
	\right\rbrace
\end{align*}
with $|\alpha| < 1, \beta >0$ and the operator $\left(\partial_x^{-1}\right)^{\ast}$ is defined in Remark \ref{remark1}.

\end{lemma}

\begin{remark}\label{remark1}
Observe that for the adjoint problem, the operator $\left(\partial_x^{-1}\right)^*$ is defined as $\partial_x^{-1} \varphi=\psi$ such that $\psi(0, y)=$ 0 and $\psi_x=\varphi$. This definition is equivalent to $$\left(\partial_x^{-1}\right)^* u(x, y, t)=\int_0^x u(s, y, t) d s.$$
\end{remark}

\begin{proof} We consider  $u\in D(A)$, $v \in D(A^\ast)$ and two functions $f, g$ such that $u_y(x,y)=f_x(x,y)$ and $v_y(x,y)=g_x(x,y)$ with $f(L,y)=g(0,y)=0$, it follows that
\begin{multline*}
-\int_{\Omega} v{\partial_x^{-1}}u_{yy}\,dx\,dy  = -\int_0^L \left[v(x,y)f(x,y)\right]_{y=0}^{y=L}dx 
  	-\int_0^L \left[g(x,y)u(x,y)\right]_{y=0}^{y=L}dx 
	\\
	+\int_{\Omega}  g_yu\,dx\,dy 
 = I+	\int_{\Omega}  u{\partial_x^{-1}}  v_{yy}\,dx\,dy,%
\end{multline*}
where $I$ is given by the first two terms of the preceding equation. By using the definition of the domains, it follows that 
\begin{equation}\label{boundaryforf}
\begin{cases}
f(x,0)=-\int_x^Lu_y(s,0)ds=0  \\
f(x,L)=-\int_x^Lu_y(s,L)ds=-\beta\int_x^Lu_x(s,L)ds =\beta u(x,L)  \\
\end{cases}
\end{equation}
and
\begin{equation}\label{boundaryforg}
\begin{cases}
g(x,0)=\int_0^xv_y(s,0)ds=0 \\
g(x,L)=\int_0^x v_y(s,L)ds=-\beta\int_0^x v_x(s,L)ds =-\beta v(x,L)  \\
\end{cases}
\end{equation}
Then, estimating $I$, we have that
\begin{align*}
I&=-\int_0^L  v(x,L)(\beta u(x,L))\,dx -\int_0^L (-\beta v(x,L))u(x,L)\,dx=0
\end{align*}
and
\begin{equation*}
\begin{aligned}
-\int_{\Omega} &v(x,y){\partial_x^{-1}}u_{yy}(x,y)\,dx\,dy 
=	\int_{\Omega}  u(x,y){\partial_x^{-1}}  v_{yy}(x,y)\,dx\,dy.% 
\end{aligned}
\end{equation*}
Consequently, we can estimate the duality product $\left\langle
 		A u, v
 	\right\rangle_{\mathcal{H}}$. Some integration by parts allows obtain
\begin{equation*}
\begin{aligned}
	\left\langle
 		A u, v
 	\right\rangle_{\mathcal{H}}  
		&=\left\langle%
 		u, A^\ast v
 	\right\rangle_{\mathcal{H}}%
\end{aligned}
\end{equation*}
Finally, note that $A^{\ast\ast}=A$ then $A$ is closed.
\end{proof}

\begin{prop}\label{pr:MUdis}
The operator $A$ is the infinitesimal generator of a $C_0$-semigroup in $\mathcal{H}$.
\end{prop}

\begin{proof}
We consider  $u\in D(A)$ and call $f_x(x,y)=u_y(x,y)$ with $f(L,y)=0$, it follows that
\begin{equation*}
\begin{aligned}
	\left\langle%
		Au,u%
	\right\rangle_\mathcal{H}
			&= \frac{1}{2}\int_0^L [u_x^2 (x,y)]_{x=0}^{x=L}dy -\int_0^L [f(x,y)u(x,y)]_{y=0}^{y=L}dx +\int_{\Omega}
			ff_x\,dx\,dy.
\end{aligned}
\end{equation*}
From \eqref{boundaryforf}, we conclude
\begin{align*}
	\left\langle%
		A u,u%
	\right\rangle_\mathcal{H}
				=- \frac{(1-\alpha^2)}{2}\int_0^L u_x^2(0,y)dy  - \beta \int_0^L u^2(x,L)dx 				-\frac{1}{2}\int_0^L(\partial_x^{-1} u_y(0,y))^2 dy . 
\end{align*}
Finally, we have that $\left\langle%
		Au,u 
	\right\rangle_\mathcal{H}
				  \leq 0.$   
Similarly, let $v\in D(A^\ast)$ and $g_x(x,y)=v_y(x,y)$ with $g(0,y)=0$. Then,
\begin{align*}
\left\langle%
	v, A^\ast v
\right\rangle_{\mathcal{H}}
			=- \frac{1}{2}\int_0^L [v_x^2 (x,y)]_{x=0}^{x=L}dy 
			 +\int_0^L[g(x,y)v(x,y)]_{y=0}^{y=L}
		dx -\int_{\Omega}
			gg_x
		\,dx\,dy 
\end{align*}
From \eqref{boundaryforg}, it yields that $\left\langle v, A^\ast v \right\rangle_{\mathcal{H}} \leq 0$. Thus, from  \cite[Corollary 4.4, page 15]{Pazy}, the Proposition holds. 
\end{proof}

Therefore, we are in a position to present the subsequent theorem concerning the existence of solutions for the Cauchy abstract problem:
\begin{thm}
For each initial data $u_0\in \mathcal{H}$, there exists a unique mild solution $u\in C\left([0,\infty), \mathcal{H}\right)$ for the system~\eqref{eq:KP}. Moreover, if the initial data $u_0\in D(A)$ the solutions are classical such that
\begin{equation*}
u\in C\left([0,\infty), D(A)\right)\cap C^1\left([0,\infty), \mathcal{H}
\right).
\end{equation*}
\end{thm}

\begin{proof}
From Proposition~\ref{pr:MUdis}, it follows that $A$ generates a strongly continuous semigroup of
contractions $\{S(t)\}_{t\geq 0}$ in $\mathcal{H}$ (see Corollary 1.4.4 in \cite{Pazy}). 
\end{proof}

Then we can establish the next proposition to state that the energy~\eqref{eq:KPen} is decreasing along the solutions of~\eqref{eq:KP}.
\begin{prop}\label{pr:Den}
For any mild solution of~\eqref{eq:KP} the energy $E(t)$, defined by \eqref{eq:KPen}, is non-increasing and there exists a constant $C>0$ such that
\begin{align}\label{eq:MUendis}
\frac{d}{dt}E(t) \leq -C
	\left[%
		\int_0^L u_{x}^2(0,y,t)\,dy+\int_0^L u^2(x,L) \,dx +\int_0^L({\partial_x^{-1}} u_y(0,y,t))^2\,dy 
	\right]
\end{align}
where $C=C(\alpha,\beta)$ is given by $C=\min
	\left\lbrace%
		\frac{(1-\alpha^2)}{2},  \beta 
	\right\rbrace.$
\end{prop}
\begin{proof}
 From proof of Proposition \ref{pr:MUdis}, \eqref{eq:MUendis} holds.
\end{proof}

The following proposition provides useful estimates for the
mild {  solutions of} \eqref{eq:KP}. The first ones are standard energy
estimates, while the last one reveals a Kato smoothing effect.

\begin{prop}
Let $u_0 \in \mathcal{H}$, then the mild solution $u$ of system \eqref{eq:KP} satisfies 
\begin{equation}\label{kato}
\|u\|_{L^2(0,T,H_x^1(\Omega))} \leq C\|u_0\|^2_{\mathcal{H}} 
\end{equation}
for some positive constant $C(L,T,\alpha,\beta)>0$. 
\end{prop}

	\begin{proof}
	To obtain the regularizing effect we employ the Morawetz multipliers technique. Indeed, by multiplying the equation \eqref{eq:KP} by $xu$ and integrating in $\Omega \times (0,T)$, we obtain that
	$$
	\int_0^T	\int_{\Omega}  xu(t) \left( u_t(t) +u_x(t)+ u_{xxx}(t)+ {\partial_x^{-1}} u_{yy}(t)  \right)  \,dx\,dy\,dt=0.
$$
	 Now, estimating term by term, it follows that the first three terms are given by
	 \begin{equation*}
	 \begin{cases}\displaystyle
	 \int_0^T	\int_{\Omega}  xu(t) u_t(t)\,dx\,dy\,dt  = \frac12 	\int_{\Omega}  x u(T)^2\,dx\,dy - \frac12 	\int_{\Omega}  x u_0^2\,dx\,dy, \\\displaystyle
	 
	 \int_0^T	\int_{\Omega}  xu(t) u_x(t)\,dx\,dy\,dt  ={  -\frac12 \int_0^T	\int_{\Omega}    u(t)^2 \,dx\,dy\,dt }\\\displaystyle
	 
	 \int_0^T	\int_{\Omega}  xu(t)  u_{xxx}(t) \,dx\,dy\,dt  
	 =\frac32 \int_0^T	\int_{\Omega}    u_{x}(t)^2 \,dx\,dy\,dt  -\frac{L}{2} \int_0^T	\int_0^L   u_{x}(L,y,t)^2\,dy\,dt  
	 \end{cases}
	 \end{equation*}	 
	 Hovewer, for the last term, we call $f_x(x,y)=u_y(x,y)$ with $f(L,y)=0$. Thus,  it follows that
\begin{multline*}
\int_0^T \int_{\Omega} x u(t)
			{\partial_x^{-1}}   u_{yy}(t)
		\,dx\,dy\,dt =\int_0^T\int_{\Omega}
		xu(t)	f_{y}(t)
		dydxdt \\
		 =\int_0^T\int_0^L[xu(t)f(t)]_{y=0}^{y=L}
		\,dx\,dt + \frac{1}{2}\int_0^T\int_{\Omega}
			f^2(t)
		\,dx\,dy\,dt.  
\end{multline*}
From \eqref{boundaryforf}, it follows that 
\begin{multline*}
\int_0^T\int_{\Omega} x u(t)
			{\partial_x^{-1}}   u_{yy}(t)
		\,dx\,dy\,dt \\
		 =\beta \int_0^T\int_0^L xu^2(x,L,t)	\,dx\,dt  +\frac{1}{2}\int_0^T\int_{\Omega}
			({\partial_x^{-1}}   u_{y}(t))^2
		\,dx\,dy\,dt.  
\end{multline*}
Now, adding the above estimates, we deduce that 
\begin{equation*}
\begin{aligned}
\frac{3}{2} \int_0^T\int_{\Omega} u_x^2(t)\,dx\,dy\,dt 
& + \frac{1}{2} \int_0^T\int_{\Omega} \left(\partial_x^{-1}u_y(t)\right)^2\,dx\,dy\,dt \\
& = \frac{1}{2}\int_{\Omega} x\left[u_0^2 -u^2(T)\right]\,dx\,dy +\frac{1}{2} \int_0^T\int_{\Omega} u^2(t)\,dx\,dy\,dt \\
 & \quad   + \frac{L}{2}\int_0^T\int_0^L u_x^2(L,y,t)\,dy\,dt - \beta \int_0^T\int_0^L xu^2(x,L,t)\,dx\,dt.
\end{aligned}
\end{equation*}
By using boundary conditions and Proposition~\ref{pr:Den},  we deduce that 
\begin{equation*}
\begin{aligned}
 \frac{3}{2} \int_0^T \int_{\Omega} u_{x}^2 \,dx\,dy\,dt	 &
%    & + \frac{1}{2} \int_0^T \int_{\Omega} \left({\partial_x^{-1}}u_y\right)^2\,dx\,dy\,dt 
%    & = \frac{1}{2} \int_{\Omega}  x u_0^2\,dx\,dy - \frac{1}{2} \int_{\Omega} xu^2(T)\,dx\,dy + \frac{1}{2}\int_0^T\int_{\Omega} u^2\,dx\,dy\,dt \\
 %   & \quad + \frac{\alpha^2L}{2}\int_0^T\int_0^L u_x^2(0,y,t)\,dy\,dt -\beta \int_0^T\int_0^L x u^2(x,L,t)\,dx\,dt \\
    \leq(L+T)\|u_0\|_{\mathcal{H}}^2 + \frac{\alpha^2L}{2}\int_0^T\int_0^L u_x^2(0,y,t)\,dy\,dt \\
    & +\beta L \int_0^T\int_0^L u^2(x,L,t)\,dx\,dt \leq \left(L+T+ \frac{\max\lbrace \alpha^2, 2\beta\rbrace}{2C(\alpha,\beta)}L\right)\|u_0\|_{\mathcal{H}}^2.
\end{aligned}
\end{equation*}
Therefore, it  yields that
\begin{equation*}
\begin{aligned}
\int_0^T \| u( t)\|_{H_x^1(\Omega)}^2 dt 
	& =  \int_0^T \| u(t)\|_{L^2(\Omega)}^2 dt  + \int_0^T \| u_x( t)\|_{L^2(\Omega)}^2 dt  \\
	& \leq \left[T + \frac{2}{3} \left(L+T+ \frac{\max\lbrace \alpha^2, 2\beta\rbrace L }{2C(\alpha,\beta)}\right)\right]\|u_0\|_{\mathcal{H}}^2. 
\end{aligned}  
\end{equation*} 
Consequently~\eqref{kato} holds for $C = C(L,T, \alpha,\beta) > 0$. 
%
%On the other hand, recall that 
%\begin{equation*}
%\lVert u \rVert_{H^1(\Omega)}^2 = \lVert u\rVert_{L^2(\Omega)}^2 + \lVert u_x\rVert_{L^2(\Omega)}^2 
%	+ \lVert u_y\rVert_{L^2(\Omega)}^2
%\end{equation*}
%By using Theorem \ref{thm:Besov} can be show that there exists $C>0$ such that $
%\lVert u_y\rVert_{L^2(\Omega)} \leq C\lVert u_x\rVert _{L^2(\Omega)}.$
%Thus, we deduce from \eqref{e1}, the following estimate 
%\begin{align*}
%\int_0^T	    \|u(t)\|_{H^1(\Omega)}^2 dt & \leq CL 	\|  u_0\|^2_{\mathcal{H}} 
%+CL \int_0^T	\int_0^L   u_{x}^2(L,y,t)dydt.  
%\end{align*}
%Moreover, by using the boundary conditions and \eqref{eq:MUendis}, it follows that 
%\begin{align*}
%\int_0^T	    \|u(t)\|_{H^1(\Omega)}^2 dt &\leq  CL	\|  u_0\|^2_{\mathcal{H}} 
%+ CL\alpha^2 \int_0^T	\int_0^L   u_{x}^2(0,y,t)dydt\\
%& \leq      CL 	\|  u_0\|^2_{\mathcal{H}}     -CL\alpha^2E(T)+L\alpha^2 E(0) \\
%&\leq      CL 	\|  u_0\|^2_{\mathcal{H}}     +2CL\alpha^2 \|  u_0\|^2_{\mathcal{H}}. 
%\end{align*}
\end{proof}

The forthcoming proposition furnishes valuable estimates for the mild solutions of \eqref{eq:KP_1}, exemplifying a standard { energy estimate.}
\begin{prop}
Let $u_0 \in \mathcal{H}$, then the mild solution $u$ of system \eqref{eq:KP} satisfies 
\begin{equation}\label{trazos}
\begin{aligned}
T \|u_0\|_{\mathcal{H}}^2 & =\int_0^T	 \| u(t)\|_{\mathcal{H}}^2dt +2 \beta \int_0^T \int_0^L  (T-t)u^2(x,L,t)	\,dx\,dt  \\
 &\quad +(1-\alpha^2)\int_0^T	\int_0^L   (T-t)   u_{x}^2(0,y,t) \,dy\,dt \\ 
&\quad +\int_0^T\int_0^L (T-t)
			({\partial_x^{-1}}   u_{y}(0,y,t))^2
		\,dy \,dt  
\end{aligned}
\end{equation}
\end{prop}
\begin{proof}
Multiplying \eqref{eq:KP} by $(T-t)u$ and integrating in $\Omega \times (0,T)$, we obtain that
	\begin{align*}
	\int_0^T	\int_{\Omega}  (T-t)u(t) \left( u_t(t) + u_{x}(t)+u_{xxx}(t)+ {\partial_x^{-1}} u_{yy}(t)  \right)  \,dx\,dy\,dt=0.
	\end{align*}
	 Now, estimating term by term, it follows that 
	 \begin{align*}
	 \int_0^T	\int_{\Omega}  (T-t)u(t) u_t(t)\,dx\,dy\,dt  	&=-\frac{T}{2} 	\int_{\Omega}  u_0^2\,dx\,dy  	 +\frac12 \int_0^T	\int_{\Omega}  u^2(t)\,dx\,dy\,dt,  \\
	 \int_0^T	\int_{\Omega}  (T-t)u(t)  u_{x}(t) \,dx\,dy\,dt  
	 &= 0
	 \end{align*}
and the another terms are given by 
	  \begin{align*}
	 \int_0^T	\int_{\Omega}  (T-t)u(t)  u_{xxx}(t) \,dx\,dy\,dt  
	 =\frac{(1-\alpha^2)}{2} \int_0^T	\int_0^L   (T-t)   u_{x}^2(0,y,t)\,dy\,dt 
	 \end{align*}
	 On the other hand, note that calling $f_x(x,y)=u_y(x,y)$ with $f(L,y)=0$ and from \eqref{boundaryforf}, it follows that
\begin{multline*}
\int_0^T\int_{\Omega} (T-t) u(t)
			{\partial_x^{-1}}   u_{yy}(t)
		\,dx\,dy dt  
		= \beta \int_0^T \int_0^L  (T-t)u^2(x,L,t)	dxdt  
		\\
		+\frac{1}{2}\int_0^T\int_0^L (T-t)
			({\partial_x^{-1}}   u_{y}(0,y,t))^2
		dy dt.  
\end{multline*}
Now, adding the above estimates, we deduce \eqref{trazos}.
	\end{proof}

\section{Exponential Stabilization - Compactness Uniqueness approach}\label{s:3}

In this section, we study the long-term behavior of solutions to \eqref{eq:KP}. Our goal is to demonstrate that the energy associated with the linear KP equation decays exponentially.

\begin{proof}[\bf Proof of Theorem \ref{11th:ESLyap_2}]
By leveraging energy dissipation, specifically the estimate \eqref{eq:MUendis}, in conjunction with a classical argument, our task simplifies to proving the \emph{observability inequality}, that is, there exists a positive constant $C$ such that
\begin{equation}\label{observability}
\begin{aligned}
\|u_0\|_{\mathcal{H}}^2 &\leq C
	\left[%
		\int_0^T\int_0^L u_{x}^2(0,y,t)\,dy\,dt
		+ \int_0^T\int_0^L u^2(x,L,t)\,dx\,dt\right. \\
		&\hspace*{4cm}\left. + \int_0^T\int_0^L({\partial_x^{-1}} u_y(0,y,t))^2\,dy\,dt 
	\right]
\end{aligned}
\end{equation}
where $u$ denotes the solution of \eqref{eq:KP}. { Indeed}, suppose that \eqref{observability} is verified and, as the energy is non-increasing, we have, thanks to \eqref{eq:MUendis}, that
$$
\begin{aligned}
E(T)- E(0) & \leq  -C
	\left[%
		\int_0^T\int_0^L u_{x}^2(0,y,t)\,dy\,dt+\int_0^T\int_0^L u^2(x,L) \,dx\,dt \right. \\
		& \qquad \qquad\qquad  \left.  +\int_0^T\int_0^L({\partial_x^{-1}} u_y(0,y,t))^2\,dy\,dt
	\right]  \leq - CE(0) \\
&\leq -C E(T) 
\end{aligned}
$$

which implies that

$$
E(T) \leq \gamma E(0), \text { with } \gamma=\frac{C}{1+C}<1
$$

The same argument used on the interval $[(m-1) T, m T]$ for $m=1,2, \ldots$, yields that

$$
E(m T) \leq \gamma E((m-1) T) \leq \cdots \leq \gamma^m E(0)
$$

Thus, we have

$$
E(m T) \leq e^{-\theta m T} E(0)\quad\text{{ with}}\quad
\theta=\frac{1}{T} \ln \left(1+\frac{1}{C}\right)>0 .
$$

For an arbitrary positive $t$, there exists $m \in \mathbb{N}^*$ such that $(m-1) T<t \leq m T$, and by the non-increasing property of the energy, we conclude that

$$
E(t) \leq E((m-1) T) \leq e^{-\theta(m-1) T} E(0) \leq \frac{1}{\gamma} e^{-\theta t} E(0) .
$$

By the density of $D(A)$ in $\mathcal{H}$, we deduce that the exponential decay of the energy $E$ holds for any initial data in $\mathcal{H}$, showing so Theorem \eqref{11th:ESLyap_2}.

\end{proof}

Let us now prove the inequality \eqref{observability}.

\begin{proof}[\bf Proof of Observability Inequality \eqref{observability}]
To establish the observability inequality \eqref{observability}, we approach it in three distinct steps:
\vglue 0.4cm
\subsection*{STEP 1. (Compactness-Uniqueness Argument)}
We argue by contradiction, applying the compactness uniqueness
argument due to E. Zuazua (see \cite{Lions1}). If \eqref{observability} is false, then we may find a sequence $\{u^n_0\}_{n\in \N}$
in $\mathcal{H}$ such that $\|u_0^n\|_{\mathcal{H}}=1$ and
\begin{equation}\label{observabilitycontra}
\begin{aligned}
\|u_0^n\|^2_{\mathcal{H}} \geq n
	\left[%
		\int_0^T\int_0^L (u^n_{x})^2(0,y,t)\,dy\,dt
		 +\int_0^T \int_0^L (u^n)^2(x,L,t) \,dx\,dt 
		\right. \\
		\left.+ \int_0^T\int_0^L({\partial_x^{-1}} u^n_y(0,y,t))^2 \,dy\,dt
	\right].
\end{aligned}
\end{equation}
On the other hand, it follows from \eqref{kato} that $\{u^n\}_{n\in\N} = \{{S}(\cdot)u_0^n\}_{n\in\N}$ is bounded in $L^2(0,T;\mathcal{H})$. By \eqref{eq:KP}, $\{u^n_t\}_{n\in\N}$ is bounded in $L^2(0,T;H^{-2}(\Omega))$, Indeed, by \eqref{eq:KP}, $$u_t^n = -u_x^n- u_{xxx}^n -  \partial_x^{-1} u_{yy}^n.$$  As $\left\lbrace u^n \right\rbrace_n$ is bounded in $L^2(0,T; H_x^1(\Omega))$ implies that $\left\lbrace u_{xxx}^n \right\rbrace_{n}$ is bounded in $L^2(0,T; H^{-2}(\Omega))$. 
%Moreover, we have that $\left\lbrace u_x^n \right\rbrace_{n}$ and $\left\lbrace u_y^n \right\rbrace_n$ are bounded in $L^2(0,T; L^2(\Omega))$. 
Define $f$ such that 
$u_y^n = f_x^n$ which implies $\partial_x^{-1} u_{yy}^n = f_y^n.$
Therefore, by using the anisotropic Gagliardo-Niremberg inequality  (Theorem \ref{thm:Besov}), it yields that  $$\lVert \partial_x^{-1} u_{yy}^n \rVert_{\mathcal{H}} = \lVert f_y^n \rVert_{\mathcal{H}} \leq C \lVert f_x^n\rVert_{\mathcal{H}} = C \lVert u_y^n \rVert_{\mathcal{H}} \leq C^2\|u_x^n\|_{\mathcal{H}} <\infty.$$ 
Recognizing $\mathcal{H}$ as the pivot space such that $H_0^1(\Omega) \subset \mathcal{H} \subset H^{-1}(\Omega)$, we select $\xi$ with sufficient regularity to ensure that
\begin{align*}
\left\lvert\left\langle \partial_x^{-1} u_{yy}^n, \xi \right\rangle_{H^{-1}, H_0^1}\right\rvert  
&= \left\lvert\left\langle f_y^n, \xi \right\rangle_{H^{-1}, H_0^1}\right\rvert  
 = \left\lvert\left\langle \xi, f_y^n \right\rangle_{\mathcal{H},\mathcal{H}}\right\rvert \\
 &  \leq \lVert \xi \rVert_{\mathcal{H}}\cdot \lVert f_y^n\rVert_{\mathcal{H}} \leq C \lVert \xi \rVert_{\mathcal{H}}\cdot \lVert f_x^n\rVert_{\mathcal{H}} 
 = C \lVert \xi \rVert_{\mathcal{H}}\cdot \lVert u_y^n\rVert_{\mathcal{H}}  .
\end{align*}
Then,
\begin{align*}
\lVert \partial_x^{-1} u_{yy}^n \rVert_{L^2(0,T; H^{-1})}^2 
 &\leq \int_0^T \lVert \partial_x^{-1} u_{yy}^n (\cdot,\cdot,t) \rVert_{H^{-1}(\Omega)}^2\,dt
  \leq C^2 \int_0^T \lVert u_y^n(\cdot,\cdot,t) \rVert_{L^2(\Omega)}^2 \\ 
 & \leq C^2 \lVert u^n \rVert_{L^2(0,T; H_x^1(\Omega))}.
\end{align*}
Hence, applying Aubin's lemma, we see that a subsequence of $\{u^n\}_{n\in\N}$, again denoted by $\{u^n\}_{n\in\N}$, converges strongly in $L^2(0,T;\mathcal{H})$ towards some $u$. Moreover, using \eqref{trazos} and \eqref{observabilitycontra}, we see that $\{u_0^n\}_{n\in\N}$ is a Cauchy sequence in $\mathcal{H}$, hence { converges} for some $u_0 \in \mathcal{H}$. It means that 
$u_0^n \rightarrow u_0$ in $\mathcal{H}$.
Clearly, $u={S}(\cdot)u_0$, and we infer from \eqref{observabilitycontra} that
\begin{equation}\label{compactuniquesscondition}
u_{x}(0,y,t)=u(x,L,t)=\partial_x^{-1} u_y(0,y,t)=0 \quad  \text{for all $x,y \in [0,L]$}
\end{equation}
and  that $\|u_0\|_{\mathcal{H}}=1$.

\subsection*{STEP 2. (Reduction to a Spectral Problem)}

{ We reduce the proof of the observability inequality into a spectral problem as done in~\cite{Rosier}  for the Korteweg-de
Vries equation}
\begin{lemma}\label{le:RedSpecP}
For any $T > 0$, let $N_T$ denote the space of all the initial states,  $u_0 \in \mathcal{H}$ for which the solution $u(t)=S(t)u_0$ of \eqref{eq:KP} satisfies \eqref{compactuniquesscondition}. Then $N_T =\{0\}$. 

\begin{proof}
The proof uses the same arguments as those given in \cite[Lemma 3.4]{Rosier}, Therefore, if $N_T \neq \emptyset$, the map $u_0 \in N_T \mapsto {A}(N_T) \subset \mathbb{C}N_T$ (where $\mathbb{C}N_T$ denote the complexification of $N_T$) has (at least) one eigenvalue; hence, there exist $\lambda \in \mathbb{C}$ and  $u_0 \in H^3_x(\Omega)\cap H^2_y(\Omega)$ such that
\begin{equation}\label{spectralsystem}
\begin{cases}
		\lambda u_0(x,y) + u_{x,0}(x,y)+ u_{xxx,0}(x,y) + \partial_x^{-1} u_{yy,0}(x,y)	=0,\\
	u_0(0,y)=u_0(L,y)=0,\quad  u_{y,0}(x,L)=\beta u_{x,0}(x,L),   \\
	u_{x,0}(L,y)=\alpha u_{x,0}(0,y)=0, \quad 	u_{y,0}(x,0)=0 \\
	u_{x,0}(0,y)=u_0(x,L)=\partial_x^{-1} u_{y,0}(0,y)=0
\end{cases}
\end{equation}
with $(x,y) \in \Omega$. To conclude the proof of the Lemma, we prove that this does not hold.
\end{proof}
\end{lemma}
In order to obtain the contradiction, it remains to prove that a duple $(\lambda, u_0)$ as above does not exist.

\subsection*{Step 3. (No Nontrivial Solution for the Spectral Problem)}
\begin{lemma}\label{step3}
Let $\lambda \in \C$ and $u_0 \in H^3_x(\Omega) \cap H^2_y(\Omega)$ fulfilling \eqref{spectralsystem}. Then $u_0=0$.
\end{lemma}

\begin{proof}
Let us introduce the notation $u(x,y) = u_0(x,y)$. Then we can rewrite our spectral problem~\eqref{spectralsystem} as
\begin{equation}\label{spectralsystem1}
\begin{cases}
		\lambda u(x,y) + u_{x}(x,y) +   u_{xxx}(x,y) + \partial_x^{-1} u_{yy}(x,y)	=0,\\
	u(0,y)=u(L,y)=0,\quad  u_{y}(x,L)=\beta u_{x}(x,L)=0,   \\ 
	u_{x}(L,y)=\alpha u_{x}(0,y)=0, \\
	 	u_{y}(x,0)=0, \quad 	u_{x}(0,y)=u(x,L)=\partial_x^{-1} u_{y}(0,y)=0.
\end{cases}
\end{equation}
{  Assume that \(u\) is a non-trivial solution that satisfies \eqref{spectralsystem1}. Then, by multiplying by \(\bar{\lambda} u(x, y)\) and integrating over \(\Omega\), it follows that
\begin{equation}
    \label{eqnewfinal1}
    0 \leqslant \int_{\Omega}|\lambda|^{2} u^{2} dxdy=\bar{\lambda}\left(-\int_{\Omega} u u_x d x d y-\int_{\Omega} u u_{x x x} d x d y-\int_{\Omega} u \partial^{-1}_x u_{y y} d x d y\right).
\end{equation}
Note that, from the boundary conditions in \eqref{spectralsystem1}, it yields
\begin{align*}
   & -\int_{\Omega} u u_x d x d y = -\frac12 \int_0^L \left( u^2(L,y)-u^2(0,y) \right) dy=0 \\
   &-\int_{\Omega} u u_{xxx} d x d y = \frac12 \int_0^L \left( u_x^2(L,y)-u^2_x(0,y) \right) dy=0 
\end{align*}
On the other hand, let us define (as before)  $f_x(x,y)=u_y(x,y)$. Thus, $f_{xy}(x,y)=u_{yy}(x,y)$, $f_y(x,y)=\partial_x^{-1}u_{yy}(x,y)$ and $f(x,y)=\partial_x^{-1} u_y(x,y)$. It is important to recognize that the definition of the non-local operator $\partial_x^{-1}$ suggests that $f_y(L,y)=f(L,y)=0$. Then, 
\begin{align*}
    -\int_{\Omega} u \partial^{-1}_x u_{y y} d x d y  =\int_{\Omega} u_y f d x d y-\int_{0}^{L} \left( u(x,L)f(x,L)- u(x,0)f(x,0)\right) d x
\end{align*}
By using the boundary conditions and \eqref{boundaryforf}, we get
\begin{align*}
    -\int_{\Omega} u \partial^{-1}_x u_{y y} d x d y 
    & =\int_{\Omega} u_y f d x d y = \int_{\Omega} f_x f d x d y = \frac{1}{2}\int_0^L \left( f(L,y) - f(0,y)\right) dy \\
    &=- \frac{1}{2}\int_0^L  \partial_x^{-1}u_{y}^2(0,y)dy=0
\end{align*}
Therefore, \eqref{eqnewfinal1} indicates that $
\int_{\Omega}|\lambda|^{2} u^{2} dxdy=0$. This result implies $u=0$, contradicting our assumption that $u$ was a non-trivial solution.}
\end{proof}

With these Lemmas, we find $u_0 = 0,$ leading to a contradiction as $\|u_0\|_{L^2(\Omega)} = 1$. Thus, we have successfully proven exponential stabilization.
\end{proof}

\begin{remark}[\bf The model without drift term]
The drift term $u_x$ naturally arises from the derivation of the equation in fluid mechanics. On the entire domain, this drift term can be eliminated by a Galilean transform. However, this is not feasible on the half-line or bounded domains due to boundary conditions. In control theory, it is known that by combining some appropriate boundary conditions and incorporating the drift term into the model, the phenomenon of the critical set appears. This imposes certain restrictions on the length size of the domain. This represents a novel observation for the KdV equation. Following this direction, it is possible to consider the KP equation without the presence of the drift term:

 \begin{equation}\label{eq:KP_1}
\begin{cases}
		u_t(x,y,t) +  u_{xxx}(x,y,t) + \partial_x^{-1} u_{yy}(x,y,t)=0, \\ 
	u(x,y,0)=u_0(x,y)
\end{cases}
\end{equation}
with the same boundary conditions:
 \begin{equation}\label{eq:KP_2}
\begin{cases}
	u(0,y,t)=u(L,y,t)=0 \,\,\  u_y(x,L,t)=\beta u_x(x,L,t), \\
	u_x(L,y,t)=\alpha u_x(0,y,t), \,\, u_y(x,0,t)=0.
\end{cases}
\end{equation}
All the calculations work for the system \eqref{eq:KP_1}-\eqref{eq:KP_2} except for Step 3, which is related to the nontrivial solution for the spectral problem when proving the observability inequality, specifically Lemma \ref{step3}. In this case, we need to slightly modify the argument in the following way:
\vglue 0.4cm
Let us introduce the notation $u(x,y)=u_0(x,y)$ and define $v_x(x,y)=u_{y,0}(x,y)$ such that $v(L,y) = 0$. Thus, we can rewrite our system \eqref{spectralsystem} (without $u_x$) as:
\begin{equation}\label{eq:spUniq0}
\begin{cases}
		-u_y(x,y) + v_x(x,y)  = 0  \\
		 v_y(x,y) + u_{xxx}(x,y)  = -\lambda u(x,y) 
\end{cases}
\end{equation}
with boundary conditions, 
\begin{equation*}
\begin{cases}
	u(0,y)=u(L,y)=0, \\
	u(x,L) = 0,\, u_{y}(x,L) = \beta u_{x}(x,L) = 0,  \\
	u_{x}(L,y)= \alpha u_{x}(0,y)=0, \, u_{y}(x,0) = 0,\\
	v(0,y) = 0.	
\end{cases}
\end{equation*}
The existence of a solution is guaranteed by Lemma~\ref{le:RedSpecP}. Now, we focus on establishing uniqueness for system \eqref{eq:spUniq0}. We assume that both $(u_1,v_1)$ and $(u_2,v_2)$ are solutions and define $u:= u_1-u_2$ and $v:= v_1 - v_2$. We analyze two cases: one where $\lambda = 0$ and another where $\lambda \neq 0$. When $\lambda=0$, we have that $-u_y(x,y) + v_x(x,y)  = 0$ and $ v_y(x,y) + u_{xxx}(x,y)  = 0.$ We use an approach by defining an energy functional. Let us define  
\begin{equation*}
\tilde{E}(u,v)(y) := \frac{1}{2} \int_0^L \left(  v^2(x,y) + u_x^2(x,y)\right)dx, \quad y \in (0,L),
\end{equation*}
 the energy associated to the system~\eqref{eq:spUniq0}. Note that, by multiplying the first and second equations of \eqref{eq:spUniq0} by $u_{xx}(x,y)$ and by $v(x,y)$, respectively, integrating over $x$ in the interval $(0,L)$, and summing the results, we can conclude that:
\begin{align*}
\int_0^L u_{xx}(x,y) \left(-u_y(x,y) + v_x(x,y)\right) 
	+ v(x,y)\left(
	 v_y(x,y)  + u_{xxx}(x,y)
	\right) \,dx= 0.
\end{align*}
Straightforward calculations and the boundary conditions imply  that 
$$
\frac{\partial}{\partial y} \tilde{E}(u,v) =
\left[ u_x(x,y) u_y(x,y) - u_{xx}(x,y)v(x,y)\right]_{x=0}^{x=L} = 0.
$$
Then,
\begin{equation*}
\tilde{E}(u,v)(y) = \frac{1}{2} \int_0^L v^2(x,y) + u_x^2(x,y) \,dx = C, \quad y \in (0,L).
\end{equation*}
Note that, by definition
$
v(x,y) = \partial_x^{-1} u_y(x,y) = -\int_x^L u_y(s,y)\,ds.
$
Thus, by the boundary condition $u_y(x,L)=0$ follows that
\begin{equation*}
v(x,L) = -\int_x^L u_y(s,L)\,ds = 0.
\end{equation*}
Since $v(x,L)=0$ and $u_x(x,L)= 0$ we get $C = 0$. Therefore,
\begin{equation*}
v(x,y) = 0 \quad\text{and}\quad u_x(x,y) = 0,\quad \forall x,y \in (0,L).
\end{equation*}
In particular, it follows that $v_1=v_2$ and notice that when $u_x(x, y) = 0,$ it implies that $u(x, y) = g(y)$. If we differentiate with respect to $y,$ we obtain $u_y(x, y) = g_y(y)$ for all $x, y \in (0, L).$ Then, we have that
\begin{equation*}
	0 =  \partial_x^{-1} u_y(x,y) = \int_x^L u_y(s,y)\,ds
	  =	 \int_x^L g_y(y) \,ds = (L-x) g_y(y).
\end{equation*}
Consolidating all information, we deduce that $g_y(y) = 0$, implying $u(x, y)$ is constant. This follows directly from the boundary conditions $u(x, y) = 0$, which implies $u_1 = u_2$.

In the next step, let us consider the case where $\lambda \neq 0.$ Observe that by multiplying the second equation in \eqref{eq:spUniq0} by $u(x, y, t),$ employing the variable change defined in \eqref{eq:spUniq0}, and integrating by parts over the domain $\Omega,$ we derive the following result:
\begin{equation*}
\begin{aligned}
-\int_{\Omega}  \lambda u^2(t)\,dx\,dy
	 &=  \int_{\Omega} u(t)v_y(t) \,dx\,dy
		+ \int_{\Omega} u(t)u_{xxx}(t) \,dx\,dy  \\
		&= - \frac{1}{2} \int_{\Omega} \left(v^2(t)\right)_x \,dx\,dy
		- \frac{1}{2} \int_0^L \left[u_x^2(x,y,t)\right]_{x=0}^{x=L} dy \\
	& = - \frac{1}{2} \int_0^L \left[v^2(x,y,t)\right]_{x=0}^{x=L}dy = 0.
\end{aligned}
\end{equation*}
This implies $u = 0$, consequently, $u_1 = u_2$. Similarly, from the equation, we conclude $v = 0$, meaning $v_1 = v_2$. Utilizing system uniqueness and recognizing zero as a valid solution, we deduce $u_0 = 0$.
\end{remark}

\section{Futher Comments}\label{sec:final}

In this work, we studied the exponential stabilization of the KP-II equation. Our main result, Theorem \ref{11th:ESLyap_2}, ensures the exponential stability of the linearized KP-II under a feedback mechanism. However, extending this result to the nonlinear KP-II with the term 
$uu_x$  (as in KdV) is challenging due to the lack of \emph{a priori}  $L^2$-estimates, mainly because of the domain's dimension. Additionally, the compactness-uniqueness argument focuses on existence and general stability, but issues like critical lengths may arise, requiring additional feedback inputs (see~\cite{Rosier, capistrano2024}). Furtermore, the exponential stabilization with a variation of a feedback-law can be considered. For example, considering an anti-damping mechanism as the boundary time-delay, we believe that a variation of the approach introduced in \cite{Baudouin, Capistrano2023} can give positive  answers to the long time behavior of the KP-II equation under damping and anti-damping mechanism, at least for the linear system.

\bmhead{Acknowledge} 
Mu\~noz was supported by FACEPE Ph.D. scholarship number IBPG-0909-1.01/20. The second author conducted this research during their visit to Universidad Nacional de Colombia Sede Manizales and expresses gratitude for warm hospitality.

\bmhead{Ethical Approval} Not applicable.

\bmhead{Funding} Mu\~noz received financial support from the FACEPE by the Ph.D. scholarship number IBPG-0909-1.01/20.

\section*{Declarations}

\bmhead{Competing Interests} I declare that the authors have no competing interests as defined by Springer, or other
interests that might be perceived to influence the results and/or discussion reported in this paper.

\end{document}